\documentclass{amsart}
\usepackage{amssymb}
\usepackage{amsmath}
\usepackage{amsfonts}

\newtheorem{theorem}{Theorem}
\theoremstyle{plain}

\newtheorem{corollary}{Corollary}

\newtheorem{definition}{Definition}
\newtheorem{example}{Example}

\newtheorem{lemma}{Lemma}

\newtheorem{proposition}{Proposition}

\numberwithin{equation}{section}
\input{tcilatex}

\begin{document}
\title[Two-sided ideals in locally $C^{\ast }$-algebras]{A note on two-sided
ideals in locally $C^{\ast }$-algebras}
\author{Alexander A. Katz}
\address{Dr. Alexander A. Katz, Department of Mathematics and Computer
Science, St. John's College of Liberal Arts and Sciences, St. John's
University, 8000 Utopia Parkway, St. John's Hall 334-i, NY 11439, USA}
\email{katza@stjohns.edu}
\date{April 7, 2013}

\begin{abstract}
In the present note we show that if $A$ is a locally $C^{\ast }$-algebra,
and $I$ and $J$ are closed two-sided ideals in $A$, then $%
(I+J)^{+}=I^{+}+J^{+}$.
\end{abstract}

\maketitle

\section{Introduction}

Let $A$ be a $C^{\ast }$-algebra, and $I$ and $J$ be closed $^{\ast }$%
-ideals in $A$. Let $I^{+}$ (resp. $J^{+}$) denotes the set of positive
elements in $I$ (resp. $J$). In 1964 in the first French edition of \cite%
{Dixmier1977} Dixmier has formulated a problem whether or not 
\begin{equation*}
(I+J)^{+}=I^{+}+J^{+}?
\end{equation*}
Using the results of Effros \cite{Effros1963} and Kadison \cite{Kadison1951},%
\cite{Kadison1965}, St\o rmer was able in 1967 to settle this problem in
affirmative in his paper \cite{Stormer1967}. In 1968 Pedersen has noticed
that it was Combes who obtained a different proof of the aforementioned
result of St\o rmer \cite{Stormer1967} as a corrolary of Pedersen's
Decomposition Theorem for $C^{\ast }$-algebra (see \cite{Pedersen1968} for
details). In 1971 Bunce has given in \cite{Bunce1971} yet another a very
short and elegant proof of the same result of St\o rmer from \cite%
{Stormer1967}.

The Hausdorff projective limits of projective families of Banach algebras as
natural locally-convex generalizations of Banach algebras have been studied
sporadically by many authors since 1952, when they were first introduced by
Arens \cite{Arens1952} and Michael \cite{Michael1952}. The Hausdorff
projective limits of projective families of $C^{\ast }$-algebras were first
mentioned by Arens \cite{Arens1952}. They have since been studied under
various names by many authors. Development of the subject is reflected in
the monograph of Fragoulopoulou \cite{Fragoulopoulou2005}. We will follow
Inoue \cite{Inoue1971} in the usage of the name \textbf{locally }$C^{\ast }$%
\textbf{-algebras} for these algebras.

The purpose of the present notes is to extend the aforementioned result of St%
\o rmer from \cite{Stormer1967} to locally $C^{\ast }$-algebras.

\section{Preliminaries}

First, we recall some basic notions on topological $^{\ast }$-algebras. A $%
^{\ast }$-algebra (or involutive algebra) is an algebra $A$ over $%
%TCIMACRO{\U{2102} }%
%BeginExpansion
\mathbb{C}
%EndExpansion
$ with an involution 
\begin{equation*}
^{\ast }:A\rightarrow A,
\end{equation*}
such that 
\begin{equation*}
(a+\lambda b)^{\ast }=a^{\ast }+\overline{\lambda }b^{\ast },
\end{equation*}
and 
\begin{equation*}
(ab)^{\ast }=b^{\ast }a^{\ast },
\end{equation*}%
for every $a,b\in A$ and $\lambda \in $ $%
%TCIMACRO{\U{2102} }%
%BeginExpansion
\mathbb{C}
%EndExpansion
$.

A seminorm $\left\Vert .\right\Vert $ on a $^{\ast }$-algebra $A$ is a $%
C^{\ast }$-seminorm if it is submultiplicative, i.e. 
\begin{equation*}
\left\Vert ab\right\Vert \leq \left\Vert a\right\Vert \left\Vert
b\right\Vert ,
\end{equation*}

and satisfies the $C^{\ast }$-condition, i.e. 
\begin{equation*}
\left\Vert a^{\ast }a\right\Vert =\left\Vert a\right\Vert ^{2},
\end{equation*}%
for every $a,b\in A.$ Note that the $C^{\ast }$-condition alone implies that 
$\left\Vert .\right\Vert $ is submultiplicative, and in particular 
\begin{equation*}
\left\Vert a^{\ast }\right\Vert =\left\Vert a\right\Vert ,
\end{equation*}%
for every $a\in A$ (cf. for example \cite{Fragoulopoulou2005}).

When a seminorm $\left\Vert .\right\Vert $ on a $^{\ast }$-algebra $A$ is a $%
C^{\ast }$-norm, and $A$ is complete in in the topology generated by this
norm, $A$ is called a $C^{\ast }$\textbf{-algebra}.

A topological $^{\ast }$-algebra is a $^{\ast }$-algebra $A$ equipped with a
topology making the operations (addition, multiplication, additive inverse,
involution) jointly continuous. For a topological $^{\ast }$-algebra $A$,
one puts $N(A)$ for the set of continuous $C^{\ast }$-seminorms on $A.$ One
can see that $N(A)$ is a directed set with respect to pointwise ordering,
because 
\begin{equation*}
\max \{\left\Vert .\right\Vert _{\alpha },\left\Vert .\right\Vert _{\beta
}\}\in N(A)
\end{equation*}%
for every $\left\Vert .\right\Vert _{\alpha },\left\Vert .\right\Vert
_{\beta }\in N(A),$ where $\alpha ,\beta \in \Lambda ,$ with $\Lambda $
being a certain directed set.

For a topological $^{\ast }$-algebra $A$, and $\left\Vert .\right\Vert
_{\alpha }\in N(A),$ $\alpha \in \Lambda $, 
\begin{equation*}
\ker \left\Vert .\right\Vert _{\alpha }=\{a\in A:\left\Vert a\right\Vert
_{\alpha }=0\}
\end{equation*}%
is a $^{\ast }$-ideal in $A$, and $\left\Vert .\right\Vert _{\alpha }$
induces a $C^{\ast }$-norm (we as well denote it by $\left\Vert .\right\Vert
_{\alpha }$) on the quotient $A_{\alpha }=A/\ker \left\Vert .\right\Vert
_{\alpha }$, and $A_{\alpha }$ is automatically complete in the topology
generated by the norm $\left\Vert .\right\Vert _{\alpha },$ thus is a $%
C^{\ast }$-algebra (see \cite{Fragoulopoulou2005} for details). Each pair $%
\left\Vert .\right\Vert _{\alpha },\left\Vert .\right\Vert _{\beta }\in
N(A), $ such that 
\begin{equation*}
\beta \succeq \alpha ,
\end{equation*}%
$\alpha ,\beta \in \Lambda ,$ induces a natural (continuous) surjective $%
^{\ast }$-homomorphism 
\begin{equation*}
g_{\alpha }^{\beta }:A_{\beta }\rightarrow A_{\alpha }.
\end{equation*}

Let, again, $\Lambda $ be a set of indices, directed by a relation
(reflexive, transitive, antisymmetric) $"\preceq "$. Let 
\begin{equation*}
\{A_{\alpha },\alpha \in \Lambda \}
\end{equation*}%
be a family of $C^{\ast }$-algebras, and $g_{\alpha }^{\beta }$ be, for 
\begin{equation*}
\alpha \preceq \beta ,
\end{equation*}%
the continuous linear $^{\ast }$-mappings 
\begin{equation*}
g_{\alpha }^{\beta }:A_{\beta }\longrightarrow A_{\alpha },
\end{equation*}%
so that 
\begin{equation*}
g_{\alpha }^{\alpha }(x_{\alpha })=x_{\alpha },
\end{equation*}%
for all $\alpha \in \Lambda ,$ and 
\begin{equation*}
g_{\alpha }^{\beta }\circ g_{\beta }^{\gamma }=g_{\alpha }^{\gamma },
\end{equation*}%
whenever 
\begin{equation*}
\alpha \preceq \beta \preceq \gamma .
\end{equation*}%
\ Let $\Gamma $ be the collections $\{g_{\alpha }^{\beta }\}$ of all such
transformations.\ Let $A$ be a $^{\ast }$-subalgebra of the direct product
algebra%
\begin{equation*}
\dprod\limits_{\alpha \in \Lambda }A_{\alpha },
\end{equation*}%
so that for its elements 
\begin{equation*}
x_{\alpha }=g_{\alpha }^{\beta }(x_{\beta }),
\end{equation*}%
for all%
\begin{equation*}
\alpha \preceq \beta ,
\end{equation*}%
where 
\begin{equation*}
x_{\alpha }\in A_{\alpha },
\end{equation*}%
and%
\begin{equation*}
x_{\beta }\in A_{\beta }.
\end{equation*}

\begin{definition}
The $^{\ast }$-algebra $A$ constructed above is called a \textbf{Hausdorff} 
\textbf{projective limit} of the projective family 
\begin{equation*}
\{A_{\alpha },\alpha \in \Lambda \},
\end{equation*}%
relatively to the collection%
\begin{equation*}
\Gamma =\{g_{\alpha }^{\beta }:\alpha ,\beta \in \Lambda :\alpha \preceq
\beta \},
\end{equation*}%
and is denoted by%
\begin{equation*}
\underleftarrow{\lim }A_{\alpha },
\end{equation*}%
$\alpha \in \Lambda ,$ and is called the Arens-Michael decomposition of $A$.
\end{definition}

It is well known (see, for example \cite{Treves1967}) that for each $x\in A,$
and each pair $\alpha ,\beta \in \Lambda ,$ such that $\alpha \preceq \beta
, $ there is a natural projection 
\begin{equation*}
\pi _{\beta }:A\longrightarrow A_{\beta },
\end{equation*}%
defined by 
\begin{equation*}
\pi _{\alpha }(x)=g_{\alpha }^{\beta }(\pi _{\beta }(x)),
\end{equation*}%
and each projection $\pi _{\alpha }$ for all $\alpha \in \Lambda $ is
continuous.

\begin{definition}
A topological $^{\ast }$-algebra $(A,\tau )$ over $\mathbb{C}$\ \ is called
a \textbf{locally }$C^{\ast }$\textbf{-algebra} if there exists a projective
family of $C^{\ast }$-algebras 
\begin{equation*}
\{A_{\alpha };g_{\alpha }^{\beta };\alpha ,\beta \in \Lambda \},
\end{equation*}%
so that 
\begin{equation*}
A\cong \underleftarrow{\lim }A_{\alpha },
\end{equation*}%
$\alpha \in \Lambda ,$ i.e. $A$ is topologically $^{\ast }$-isomorphic to a
projective limit of a projective family of $C^{\ast }$-algebras, i.e. there
exits its Arens-Michael decomposition of $A$ composed entirely of $C^{\ast }$%
-algebras.
\end{definition}

A topological $^{\ast }$-algebra $(A,\tau )$ over $\mathbb{C}$ is a locally $%
C^{\ast }$-algebra iff $A$ is a complete Hausdorff topological $^{\ast }$%
-algebra in which the topology $\tau $ is generated by a saturated
separating family $\digamma $ of $C^{\ast }$-seminorms (see \cite%
{Fragoulopoulou2005} for details).

\begin{example}
Every $C^{\ast }$-algebra is a locally $C^{\ast }$-algebra.
\end{example}

\begin{example}
A closed $^{\ast }$-subalgebra of a locally $C^{\ast }$-algebra is a locally 
$C^{\ast }$-algebra.
\end{example}

\begin{example}
The product $\dprod\limits_{\alpha \in \Lambda }A_{\alpha }$ of $C^{\ast }$%
-algebras $A_{\alpha }$, with the product topology, is a locally $C^{\ast }$%
-algebra.
\end{example}

\begin{example}
Let $X$ be a compactly generated Hausdorff space (this means that a subset $%
Y\subset X$ is closed iff $Y\cap K$ is closed for every compact subset $%
K\subset X$). Then the algebra $C(X)$ of all continuous, not necessarily
bounded complex-valued functions on $X,$ with the topology of uniform
convergence on compact subsets, is a locally $C^{\ast }$-algebra. It is well
known that all metrizable spaces and all locally compact Hausdorff spaces
are compactly generated (see \cite{Kelley1975} for details).
\end{example}

Let $A$ be a locally $C^{\ast }$-algebra. Then an element $a\in A$ is called 
\textbf{bounded}, if 
\begin{equation*}
\left\Vert a\right\Vert _{\infty }=\{\sup \left\Vert a\right\Vert _{\alpha
},\alpha \in \Lambda :\left\Vert .\right\Vert _{\alpha }\in N(A)\}<\infty .
\end{equation*}%
The set of all bounded elements of $A$ is denoted by $b(A).$

It is well-known that for each locally $C^{\ast }$-algebra $A,$ its set $%
b(A) $ of bounded elements of $A$ is a locally $C^{\ast }$-subalgebra, which
is a $C^{\ast }$-algebra in the norm $\left\Vert .\right\Vert _{\infty },$
such that it is dense in $A$ in its topology (see for example \cite%
{Fragoulopoulou2005}).

\subsection{The cone of positive elements in a locally C*-algebra.}

If $(A,\tau )$ is a unital topological $^{\ast }$-algebra, then the \textbf{%
spectrum} $sp_{A}(a)$ of an element $a\in A$ is the set 
\begin{equation*}
sp_{A}(a)=\{z\in 
%TCIMACRO{\U{2102} }%
%BeginExpansion
\mathbb{C}
%EndExpansion
:z\mathbf{e}_{A}-a\notin G_{A}\},
\end{equation*}%
where $\mathbf{e}_{A}$ is a unital element in $A$, and $G_{A}$ is the group
of invertable elements in $A$.

An element $a\in A$ in a unital topological $^{\ast }$-algebra $(A,\tau )$
is called positive, and we write 
\begin{equation*}
a\geq \mathbf{0}_{A},
\end{equation*}%
if $a\in A_{sa},$ and 
\begin{equation*}
sp_{A}(a)\subseteq \lbrack 0,\infty ).
\end{equation*}%
We denote the set of positive elements is $(A,\tau )$ by $A^{+}.$

Let A,$\tau $ be a locally $C^{\ast }$-algebra, and 
\begin{equation*}
A\cong \underleftarrow{\lim }A_{\alpha },
\end{equation*}%
$\alpha \in \Lambda ,$ be its Arens-Michael decomposition. Then 
\begin{equation*}
a\geq \mathbf{0}_{A}\longleftrightarrow a_{\alpha }\geq \mathbf{0}%
_{A_{\alpha }},\text{ for all }\alpha \in \Lambda ,
\end{equation*}%
where $a_{\alpha }=\pi _{\alpha }(a).$

A \textbf{positive part} $A^{+}$ of any locally $C^{\ast }$-algebra is not
empty, and, together with each $a\in A_{sa},$ it contains positive elements $%
a^{+},a^{-}$ and $\left\vert a\right\vert ,$ such that 
\begin{equation*}
a=a^{+}-a^{-},
\end{equation*}%
\begin{equation*}
a^{+}a^{-}=\mathbf{0}_{A}=a^{-}a^{+},
\end{equation*}%
\begin{equation*}
\left\vert a\right\vert =a^{+}+a^{-}.
\end{equation*}

For any positive $a\in A^{+}$ in any locally $C^{\ast }$-algebra $A$ there
exists a unique \textbf{positive squire root} $b\in A^{+},$ such that 
\begin{equation*}
a=b^{2}.
\end{equation*}%
(see \cite{Fragoulopoulou2005} for details).

The following theorem is valid:

\begin{theorem}[Inoue-Schm\"{u}dgen]
Let $(A,\tau )$ be a locally $C^{\ast }$-algebra. The following are
equivalent:

\begin{enumerate}
\item $a\geq \mathbf{0}_{A};$

\item $a=b^{\ast }b,$ for some $b\in A;$

\item $a=c^{2},$ for some $c\in A_{sa}.$
\end{enumerate}
\end{theorem}

\begin{proof}
See \cite{Inoue1971} and \cite{Schmudgen1975} for details.
\end{proof}

As a corollary from this theorem one can see that $A^{+}=\{a^{\ast }a:a\in
A\}$ in each locally $C^{\ast }$-algebra $A$.

The following theorem is valid:

\begin{theorem}[Inoue-Schm\"{u}dgen]
Let $(A,\tau )$ be a locally $C^{\ast }$-algebra. Then $A^{+}$ is a closed
convex cone such that $A^{+}\cap (-A^{+})=\{\mathbf{0}_{A}\}.$
\end{theorem}

\begin{proof}
See \cite{Inoue1971} and \cite{Schmudgen1975} for details.
\end{proof}

\section{Positive parts of two-sided ideals in locally C*-algebras}

We start by recalling the following well-known result:

\begin{proposition}
Let $(A,\tau _{A})$ and $(B,\tau _{B})$ be two locally $C^{\ast }$-algebras,
and 
\begin{equation*}
\varphi :A\rightarrow B,
\end{equation*}%
be a $^{\ast }$-homomorphism. Then 
\begin{equation*}
\varphi (A^{+})=B^{+}\cap \varphi (A).
\end{equation*}
\end{proposition}

\begin{proof}
See for example \cite{Fragoulopoulou2005} for details.{}
\end{proof}

As a corrolary from that theorem one gets:

\begin{corollary}
Let $(B,\tau _{B})$ be a locally $C^{\ast }$-algebra, and $(A,\tau _{A})$ be
its closed *-subalgebra (with $\tau _{A}=\tau _{B}|_{A}$). Then 
\begin{equation*}
A^{+}=B^{+}\cap A.
\end{equation*}
\end{corollary}

\begin{proof}
See for example \cite{Fragoulopoulou2005} for details.
\end{proof}

Now, let $(A,\tau _{A})$ be an unital locally $C^{\ast }$-algebra, $I$ and $%
J $ be closed $^{\ast }$-ideals of $(A,\tau _{A})$, $B$ be a $C^{\ast }$%
-algebra, and 
\begin{equation*}
\varphi :A\rightarrow B,
\end{equation*}%
be a surjective continuous $^{\ast }$-homomorphism from $A$ onto $B$. In the
following series of lemmata we analyse what $\varphi (I),\varphi (J),\varphi
(A^{+}),\varphi (I^{+}),\varphi (J^{+}),\varphi (I^{+}+J^{+})$ and $\varphi
((I+J)^{+})$ are in $B$.

\begin{lemma}
Let $A,B,I$ $(J),\varphi $ be as above. Then $\varphi (I)$ (resp. $\varphi
(J)$) is a closed $^{\ast }$-ideal in $B$.
\end{lemma}

\begin{proof}
Because $\varphi $ is continuous surjection from $A$ onto $B$, and $I$ is a $%
^{\ast }$-subalgebra (as a $^{\ast }$-ideal), it follows that $\varphi (I)$
is a closed $^{\ast }$-subalgebra of $B$. It remained to show that if $x\in
B,$ then 
\begin{equation*}
x\cdot \varphi (I)\subset \varphi (I).
\end{equation*}%
Let $a$ be a fixed arbitrary element 
\begin{equation*}
a\in \varphi ^{-1}(x)\subset A
\end{equation*}%
(it means that $\varphi (a)=x$)$,$ and $b\in I.$ Because $I$ is an $^{\ast }$%
-ideal in $A$, it follows that 
\begin{equation*}
ab=c\in I.
\end{equation*}%
Therefore, from 
\begin{equation*}
\varphi (a)\cdot \varphi (b)=\varphi (ab)=\varphi (c)\in \varphi (I)
\end{equation*}%
it follows that 
\begin{equation*}
x\cdot \varphi (b)\in \varphi (I),
\end{equation*}%
for all $b\in I,$ Q.E.D.
\end{proof}

\begin{lemma}
Let $A,B,A^{+},B^{+}$ and $\varphi $ be as above. Then 
\begin{equation*}
\varphi (A^{+})=B^{+}.
\end{equation*}
\end{lemma}

\begin{proof}
From Proposition 1 it follows that 
\begin{equation*}
\varphi (A^{+})=B^{+}\cap \varphi (A)\subset B^{+}.
\end{equation*}%
Therefore, it is enough to show that 
\begin{equation*}
B^{+}\subset \varphi (A^{+}).
\end{equation*}%
Let $x$ be an arbitrary element in $B^{+}.$ Then (see for example \cite%
{Pedersen1979}) there exists a positive $y\in B^{+},$ such that 
\begin{equation*}
x=y^{2}.
\end{equation*}%
Let $b\in A$ be an arbitrary element in 
\begin{equation*}
\varphi ^{-1}(y)\subset A
\end{equation*}%
(it means that $\varphi (b)=y$)$.$ From Inoue-Schm\"{u}dgen Theorem 1 above
it follows that 
\begin{equation*}
b^{2}\in A^{+},
\end{equation*}%
and 
\begin{equation*}
\varphi (b^{2})=(\varphi (b))^{2}=y^{2}=x,
\end{equation*}%
Q.E.D.
\end{proof}

\begin{lemma}
Let $A,B,I$ (resp. $J$), $I^{+}$ (resp. $J^{+}$) and $\varphi $ be as above.
Then 
\begin{equation*}
\varphi (I^{+})=(\varphi (I))^{+}
\end{equation*}
(resp. $\varphi (J^{+})=(\varphi (J))^{+}$).
\end{lemma}

\begin{proof}
Because $I$ is a locally $C^{\ast }$-algebra, and $\varphi (I)$ is a $%
C^{\ast }$-algebra, the statement is immediately follows from Lemma 2 above.
\end{proof}

\begin{lemma}
Let $A,B,I,J,I^{+},J^{+}$ and $\varphi $ be as above. Then 
\begin{equation*}
\varphi (I^{+}+J^{+})=(\varphi (I))^{+}+((\varphi (J))^{+}
\end{equation*}
\end{lemma}

\begin{proof}
Because $\varphi $ is a $^{\ast }$-homomorphism, from Lemma 2 above it
follows that 
\begin{equation*}
\varphi (I^{+}+J^{+})=\varphi (I^{+})+\varphi (J^{+})=(\varphi
(I))^{+}+((\varphi (J))^{+},
\end{equation*}%
Q.E.D.
\end{proof}

\begin{lemma}
Let $A,B,I,J,I^{+},J^{+}$ and $\varphi $ be as above. Then $\ $%
\begin{equation*}
\varphi ((I+J)^{+})=(\varphi (I)+\varphi (J))^{+}
\end{equation*}
\end{lemma}

\begin{proof}
Let $c$ be an arbitrary element in $(I+J)^{+}.$ From Inoue-Schm\"{u}dgen
Theorem 1 above it follows that there exist $a\in I$ and $b\in J,$ such that 
\begin{equation*}
c=(a+b)^{\ast }(a+b).
\end{equation*}%
Let us denote 
\begin{equation*}
x=\varphi (a),\text{ and }y=\varphi (b).
\end{equation*}%
Then 
\begin{equation*}
\varphi (c)=\varphi ((a+b)^{\ast }(a+b))=\varphi ((a+b)^{\ast })\varphi
(a+b)=
\end{equation*}%
\begin{equation*}
=(\varphi (a)+\varphi (b))^{\ast }(\varphi (a)+\varphi (b))\in (\varphi
(I)+\varphi (J))^{+},
\end{equation*}%
thus 
\begin{equation*}
\varphi ((I+J)^{+})\subset (\varphi (I)+\varphi (J))^{+}.
\end{equation*}

Inversly, let $t$ be an arbitrary element 
\begin{equation*}
t\in (\varphi (I)+\varphi (J))^{+}.
\end{equation*}%
From the basic properties of the cone of positive elements in a $C^{\ast }$%
-algera (see for example \cite{Pedersen1979}) if follows that there exist 
\begin{equation*}
x\in \varphi (I),\text{ and }y\in \varphi (J),
\end{equation*}%
such that 
\begin{equation*}
t=(x+y)^{\ast }(x+y).
\end{equation*}%
Let $a$ be an arbitrary element in 
\begin{equation*}
\varphi ^{-1}(x)\subset I,
\end{equation*}%
and $b$ be an arbitrary element in 
\begin{equation*}
\varphi ^{-1}(y)\subset J.
\end{equation*}%
It means that 
\begin{equation*}
x=\varphi (a),\text{ and }y=\varphi (b).
\end{equation*}%
Then from Inoue-Schm\"{u}dgen Theorem 1 above it follows that the element 
\begin{equation*}
(a+b)^{\ast }(a+b)\in (I+J)^{+},
\end{equation*}%
and 
\begin{equation*}
\varphi ((a+b)^{\ast }(a+b))=\varphi ((a+b)^{\ast })\varphi (a+b)=(\varphi
(a)+\varphi (b))^{\ast }(\varphi (a)+\varphi (b))=
\end{equation*}%
\begin{equation*}
=(x+y)^{\ast }(x+y)=t,
\end{equation*}%
thus 
\begin{equation*}
\varphi ((I+J)^{+})\supset (\varphi (I)+\varphi (J))^{+},
\end{equation*}%
Q.E.D.
\end{proof}

Now we are ready to present the main theorem of the current notes.

\begin{theorem}
Let $(A,\tau _{A})$ be an unital locally $C^{\ast }$-algebra, and $(I,\tau
_{I})$ and $(J,\tau _{J})$ be two $^{\ast }$-ideals in $A,$ such that 
\begin{equation*}
\tau _{I}=\tau _{A}|_{I}\text{ and }\tau _{J}=\tau _{A}|_{J}.
\end{equation*}%
Then 
\begin{equation*}
(I+J)^{+}=I^{+}+J^{+}.
\end{equation*}
\end{theorem}

\begin{proof}
Let now $(A,\tau )$ be a locally $C^{\ast }$-algebra, and let 
\begin{equation*}
A=\underleftarrow{\lim }A_{\alpha },
\end{equation*}%
$\alpha \in \Lambda $, be its Arens-Michael decomposition, built using the
family of seminorms $\left\Vert .\right\Vert _{\alpha },\alpha \in \Lambda ,$
that defines the topology $\tau .$ Let 
\begin{equation*}
\pi _{\alpha }:A\rightarrow A_{\alpha },
\end{equation*}%
$\alpha \in \Lambda ,$ be a projection from $A$ onto $A_{\alpha },$ for each 
$\alpha \in \Lambda .$ Each $\pi _{\alpha }$ is an injective $^{\ast }$%
-homomorphism from $A$ onto $A_{\alpha },\alpha \in \Lambda ,$ thus, we can
apply to $A,$ $A_{\alpha },$ and $\pi _{\alpha }$ lemmata 3-5 above for all $%
\alpha \in \Lambda .$ Let 
\begin{equation*}
I_{\alpha }=\pi _{\alpha }(I)
\end{equation*}%
(resp. $J_{\alpha }=\pi _{\alpha }(J)$). Therefore, for all $\alpha \in
\Lambda ,$ St\o rmer's result from \cite{Stormer1967} for $C^{\ast }$%
-algebras implies that%
\begin{equation*}
(I_{\alpha }+J_{\alpha })^{+}=I_{\alpha }^{+}+J_{\alpha }^{+}.
\end{equation*}%
Thus, because 
\begin{equation*}
(I+J)^{+}=\underset{\longleftarrow }{\lim }(I+J)_{\alpha }^{+},
\end{equation*}%
\begin{equation*}
I^{+}=\underset{\longleftarrow }{\lim }I_{\alpha }^{+},
\end{equation*}%
and 
\begin{equation*}
J^{+}=\underset{\longleftarrow }{\lim }J_{\alpha }^{+},
\end{equation*}%
$\alpha \in \Lambda ,$ we get 
\begin{equation*}
(I+J)^{+}=I^{+}+J^{+},
\end{equation*}%
Q.E.D.
\end{proof}

\end{document}